\documentclass[12pt,reqno]{amsart}
\usepackage{amsmath, amsthm, amssymb}

\advance \topmargin by -\headheight
\advance \topmargin by -\headsep
     
\evensidemargin \oddsidemargin
\newtheorem{theorem}{Theorem}[section]
\newtheorem{lemma}[theorem]{Lemma}
\newtheorem{corollary}[theorem]{Corollary}

\theoremstyle{definition}
\newtheorem{definition}[theorem]{Definition}

\theoremstyle{remark}

\numberwithin{equation}{section}

\def\bfa{{\mathbf a}}
\def\bfb{{\mathbf b}}
\def\bfc{{\mathbf c}}

\def\bfx{{\mathbf x}}

\def\calA{{\mathcal A}}  

\def\calB{{\mathcal B}}

\def\Hhat{\widehat H}

\def\dbC{{\mathbb C}}
\def\dbR{{\mathbb R}}
\def\dbZ{{\mathbb Z}}

\def\grB{{\mathfrak B}}

\def\grJ{{\mathfrak J}}

\def\grm{{\mathfrak m}}\def\grM{{\mathfrak M}}\def\grN{{\mathfrak N}}
\def\grn{{\mathfrak n}}\def\grS{{\mathfrak S}}\def\grP{{\mathfrak P}}
\def\grB{{\mathfrak B}}
\def\grK{{\mathfrak K}}

\def\grv{{\mathfrak v}}\def\grV{{\mathfrak V}}

\def\alp{{\alpha}} \def\bfalp{{\boldsymbol \alpha}}
\def\bet{{\beta}}  \def\bfbet{{\boldsymbol \beta}}
\def\gam{{\gamma}} 
\def\del{{\delta}}

\def\tet{{\theta}} \def\bftet{{\boldsymbol \theta}} 
\def\kap{{\kappa}}
  
\def\bfnu{{\boldsymbol \nu}}

\def\sig{{\sigma}}  \def\bfsig{{\boldsymbol \sig}}

\def\ome{{\omega}}  \def\bfome{{\boldsymbol \ome}}
\def\d{{\partial}}
\def\eps{\varepsilon}

\def\le{\leqslant} \def\ge{\geqslant}

\def\d{{\,{\rm d}}}

\begin{document}
\title[Systems of diagonal cubic forms]{The Hasse principle for systems\\ of diagonal cubic forms}
\author[J\"org Br\"udern]{J\"org Br\"udern}
\address{Mathematisches Institut, Bunsenstrasse 3--5, D-37073 G\"ottingen, Germany}
\email{bruedern@uni-math.gwdg.de}
\author[Trevor D. Wooley]{Trevor D. Wooley}
\address{School of Mathematics, University of Bristol, University Walk, Clifton, Bristol BS8 1TW, United Kingdom}
\email{matdw@bristol.ac.uk}
\thanks{The authors are grateful to the Hausdorff Research Institute for Mathematics in 
Bonn for excellent working conditions that made the writing of this paper feasible. The 
support of the Akademie der Wissenschaften zu G\"ottingen is also gratefully 
acknowledged.}
\date{}
\begin{abstract} We establish the Hasse Principle for systems of $r$ simultaneous diagonal cubic equations
 whenever the number of variables exceeds $6r$ and the associated coefficient matrix contains no singular
 $r\times r$ submatrix, thereby achieving the theoretical limit of the circle method for such systems.
\end{abstract}
\maketitle

\section{Introduction} The Diophantine analysis of systems of diagonal equations was pioneered by
 Davenport and Lewis with a pivotal contribution on pairs of cubic forms \cite{DL1966}, followed by work
 on more general systems \cite{DL1969}. For natural numbers $r,s$ and an $r\times s$ integral matrix
 $(c_{ij})$, they applied the circle method to the system
\begin{equation}\label{1.1}
\sum_{j=1}^sc_{ij}x_j^3=0\quad (1\le i\le r),
\end{equation}
and when $s\ge 27r^2\log 9r$ were able to show that
 (\ref{1.1}) has infinitely many primitive integral solutions. Even a casual practitioner in the field will
 acknowledge that the implicit use of mean values demands at least $6r+1$ variables in the system for the
 circle method to be applicable. We now attain this theoretical limit, surmounting the obstacles encountered
 by previous writers.

\begin{theorem}\label{theorem1.1} Let $s>6r$ and suppose that the matrix $(c_{ij})$ contains no singular
 $r\times r$ submatrix. Then, whenever the system (\ref{1.1}) has non-zero $p$-adic solutions for all 
primes $p$, it has infinitely many primitive integral solutions.
\end{theorem}

The conclusion of Theorem \ref{theorem1.1} may be interpreted as a Hasse principle for
 systems of diagonal cubic forms in general position. As we remark in \S4, the condition on the matrix of
 coefficients can be relaxed considerably. Should the local solubility conditions be met, our methods show
 that the number $N(P)$ of integral solutions of (\ref{1.1}) with $\bfx\in [-P,P]^s$ satisfies 
$N(P)\gg P^{s-3r}$. We note that work of the first author joint with Atkinson and Cook \cite{ABC1992}
 implies that for $p>9^{r+1}$ the $p$-adic solubility hypothesis in Theorem \ref{theorem1.1} is void.
\par

Early work on this subject concentrated on methods designed to disentangle the system so as to invoke
 results on single equations. The most recent such contribution is Br\"udern and Cook \cite{BC1992} where 
the condition $s>7r$ is imposed on the number of variables. Such methods are incapable of establishing the 
conclusion of Theorem \ref{theorem1.1} unless one is prepared to invoke conditional mean value estimates
 that depend on speculative Riemann hypotheses for global Hasse-Weil $L$-functions (see Hooley 
\cite{Hoo1986, Hoo1996} and Heath-Brown \cite{HB1998}).\par

When $r=1$, the conclusion of Theorem \ref{theorem1.1} is due to Baker \cite{Bak1989}. For $r\ge 2$,
 the present authors \cite{BW2002} identified features of fully entangled systems of equations which permit 
highly efficient use of divisor estimates in bounding associated multidimensional mean values. These allow 
treatment of systems in $6r+3$ variables. By a method special to the case $r=2$, we established that case 
of Theorem \ref{theorem1.1} in more general form (see \cite{BW2007}). In this paper we instead develop 
a recursive process that relates mean values associated with the original system to a one-dimensional sixth 
moment of a smooth Weyl sum on the one hand, and on the other to another system of the shape 
(\ref{1.1}), but of much larger format. The new system is designed in such a way that the methods of 
\cite{BW2002} provide very nearly square-root cancellation. By comparison with older routines, we are
 forced to incorporate the losses implied by the use of a sixth moment of a smooth Weyl sum only once, as 
opposed to $r$ times (in \cite{BC1992}, for example).\par

Our recursive process relies on an analytic inequality that is simple to describe. Suppose that $1\le r<R$, 
and that $G(\alp_1,\ldots ,\alp_r)$ and $F(\alp_1,\ldots ,\alp_R)$ are exponential sums, and consider the 
integral
$$I(F,G)=\int_0^1\ldots \int_0^1 G(\alp_1,\ldots ,\alp_r)F(\alp_1,\ldots ,\alp_R)\d\alp_1 \ldots \d\alp_R .
$$
Then by Schwarz's inequality, one finds that $I(F,G)^2\le I_1I_2$, with
\begin{equation}\label{1.2}
I_1=\int_0^1\ldots \int_0^1\biggl| \int_0^1\ldots \int_0^1 F(\alp_1,\ldots ,\alp_R)\d\alp_{r+1}\ldots 
\d\alp_R \biggr|^2\d\alp_1\ldots \d\alp_r
\end{equation}
and
\begin{equation}\label{1.3}
I_2=\int_0^1\ldots \int_0^1 |G(\alp_1,\ldots ,\alp_r)|^2\d\alp_1\ldots \d\alp_r.
\end{equation}
In our application of this inequality, the integral $I(F,G)$ will count the number of solutions of a system of 
$R$ linear equations to be solved in integral cubes. We shall take $r=1$ and $G(\alp_1)$ equal to an 
exponential sum related to sums of three cubes. Then, the mean square (\ref{1.3}) is a sixth moment of 
cubic Weyl sums for which strong bounds are available. Also, on opening the square in $I_1$, a Diophantine 
interpretation of (\ref{1.2}) with $2R-r$ equations is induced. It transpires that this procedure can be 
repeated, achieving a satisfactory bound for $I(F,G)$ whenever a good bound for the mean square 
(\ref{1.3}) is partnered with good control for the high-dimensional iterates of $I_1$ that arise from the 
recursion. While inspired by the work of Gowers \cite{Gow2001}, the procedure sketched here is in 
principle very flexible. For example, variants may be developed involving higher moments.\par

This paper is organised as follows. We begin in \S2 by describing the linked block matrices underpinning
 our new mean value estimates. By using an argument motivated by our earlier work \cite{BW2002}, we
 derive strong estimates associated with Diophantine systems having six times as many variables as
 equations. Next, in \S3, by repeated application of Schwarz's inequality, we transform an initial system of
 equations into a more complicated system of the type just analysed. Thus, a powerful mean value estimate
 is obtained that leads in \S4 via the circle method to the proof of Theorem \ref{theorem1.1}.\par

Our basic parameter is $P$, a sufficiently large positive number. In this paper, implicit constants in
 Vinogradov's notation $\ll$ and $\gg$ may depend on $s$, $r$ and $\eps$, as well as ambient coefficients.
 Whenever $\eps$ appears in a statement, either implicitly or explicitly, we assert that the statement holds
 for each $\eps>0$. We employ the convention that whenever $G:[0,1)^k\rightarrow \dbC$ is integrable,
 then
$$\oint G(\bfalp)\d\bfalp =\int_{[0,1)^k}G(\bfalp)\d\bfalp .$$
Here and elsewhere, we use vector notation in the natural way. Finally, we write $e(z)$ for $e^{2\pi iz}$
 and put $\|\tet\|=\min\{|\tet-m|:m\in \dbZ\}$.

\section{Auxiliary equations} We begin by defining a strong form of non-singularity satisfied by almost all
 coefficient matrices. We refer to an $r\times s$ matrix $A$ as {\it highly non-singular} when any subset of 
at most $r$ columns of $A$ is linearly independent. For example, the matrix
$$B={\tiny{\left( \arraycolsep=1.6pt\begin{array}{*{9}c}
7&1&4&8&8&4&9&8&1\\
7&5&6&3&3&7&1&7&8\\
9&4&5&7&1&6&5&3&6\\
6&3&3&8&8&6&9&9&3\end{array}\right)}}$$
is highly non-singular, as the reader may care to check.

\begin{lemma}\label{lemma2.1} Suppose that the matrix $A$ is highly non-singular. Then the submatrix  
obtained by deleting a column is highly non-singular. Also, if a column of $A$ contains just one non-zero 
element, then the submatrix obtained by deleting the column and row containing this element is highly 
non-singular.\end{lemma}

\begin{proof} Both conclusions follow from the definition of highly non-singular.\end{proof}

Next we describe linked block matrices critical to our arguments. Even to describe the 
shape of these matrices takes some effort. When $n$ is a non-negative integer and 
$0\le l\le n$, consider natural numbers $r_l,s_l$ and an $r_l\times s_l$ matrix $A_l$ 
having non-zero columns. Let $\text{diag}(A_0,A_1,\ldots ,A_n)$ be the conventional 
diagonal block matrix with the lower right hand corner of $A_l$ sited at $(i_l,j_l)$. For 
$1\le l\le n$, append a row to the top of the matrix $A_l$, giving an $(r_l+1)\times s_l$ 
matrix $B_l$. Next, consider the matrix $D=(d_{ij})$ obtained from 
$\text{diag}(A_0,\ldots ,A_n)$ by replacing  $A_l$ by $B_l$ for $1\le l\le n$, with the 
lower right hand corner of $B_l$ still sited at $(i_l,j_l)$. This new {\it linked-block} matrix 
$D$ should be thought of as a matrix with additional entries by comparison to 
$\text{diag}(A_0,\ldots ,A_n)$, with the property that adjacent blocks are glued together 
by a shared row sited at index $i_l$, for $0\le l<n$.

\begin{definition}\label{definition2.2} We say that the linked block matrix $D$ is {\it congenial of type} 
$(n,r;\rho,u,t)$ when it has the shape described above, and
\begin{itemize}
\item[(a)] $A_l$ and $B_l$ are highly non-singular, with $B_l$ of format $r\times 3(r-1)$, 
for $1\le l\le n$;
\item[(b)] $A_0$ is a $\rho \times t$ matrix having the following properties:
\begin{itemize}
\item[(i)] when $\rho\ge 2$, its first $u$ columns define a subspace of dimension $1$ 
distinct from the $\rho$-th coordinate axis;
\item[(ii)] the matrix of its last $t-u+1$ columns is highly non-singular;
\item[(iii)] if $u\ge 3$, then $t\ge 3\rho$.
\end{itemize}
\end{itemize}
\end{definition}

As a helping hand to the reader, we illustrate this definition with an example. Thus the 
matrix\footnote{Henceforth we adopt the convention that zero entries in a matrix are left 
blank.} 
$${\tiny{\left( \arraycolsep=1.6pt\begin{array}{*{36}c}
1&3&3&3&7&1&7&8&&&&&&&&&&&&&&&&&&&&&&&&\\
&&7&1&6&5&3&6&&&&&&&&&&&&&&&&&&&&&&&&\\
&&8&8&6&9&9&3&7&1&4&8&8&4&9&8&1&&&&&&&&&&&&&&&&\\
&&&&&&&&7&5&6&3&3&7&1&7&8&&&&&&&&&&&&&&&&\\
&&&&&&&&9&4&5&7&1&6&5&3&6&&&&&&&&&&&&&&&&\\
&&&&&&&&6&3&3&8&8&6&9&9&3&7&1&4&8&8&4&9&8&1&&&&&&&&\\
&&&&&&&&&&&&&&&&&7&5&6&3&3&7&1&7&8&&&&&&&&\\
&&&&&&&&&&&&&&&&&9&4&5&7&1&6&5&3&6&&&&&&&&\\
&&&&&&&&&&&&&&&&&6&3&3&8&8&6&9&9&3&7&1&4&8&8&4&9&8&1\\
&&&&&&&&&&&&&&&&&&&&&&&&&&7&5&6&3&3&7&1&7&8\\
&&&&&&&&&&&&&&&&&&&&&&&&&&9&4&5&7&1&6&5&3&6\\
&&&&&&&&&&&&&&&&&&&&&&&&&&6&3&3&8&8&6&9&9&3
\end{array}\right) }}$$
is congenial of type $(3,4;3,2,8)$. In terms of the description above, one sees that
$$A_0={\tiny{ \left( \arraycolsep=1.6pt\begin{array}{*{8}c}
1&3&3&3&7&1&7&8\\
&&7&1&6&5&3&6\\
&&8&8&6&9&9&3\end{array}\right) }}$$
and $B_1=B_2=B_3=B$, and further $A_1=A_2=A_3=A$, with
$$A={\tiny{\left( \arraycolsep=1.6pt\begin{array}{*{9}c}
7&5&6&3&3&7&1&7&8\\
9&4&5&7&1&6&5&3&6\\
6&3&3&8&8&6&9&9&3\end{array}\right)}}.$$

Some additional remarks are in order to clarify this definition. With an inductive argument 
in mind, we allow the possibility that $n=0$, in which case the parameter $r$ plays no 
role. Note that when $n\ge 1$, the definition is non-empty only when $r\ge 2$. In 
preparation for our inductive argument, once again, we allow the possibility that $A_0$ is 
the empty matrix formally considered to have format $1\times 0$, and we accommodate 
this situation by identifying congenial matrices (formally) of type $(n,r;1,0,0)$ with those 
of type $(n-1,r;r,1,3r-3)$. Here, as we shall see, there is no loss of generality in assuming 
that the columns in the first non-empty block of $D$ have been permuted in order to 
ensure that its first column is distinct from the $r$-th coordinate axis. When $t\ge 1$, we 
insist that $u\ge 1$, consistent with the hypothesis that the last $t-u+1$ columns of $A_0$ 
be highly non-singular. Also, we note that when $\rho=1$, the conditions imposed in the 
preamble to Definition \ref{definition2.2} require that $d_{1j}\ne 0$ for $1\le j\le t$, and 
(b) is then satisfied for all $1\le u\le t$. When $\rho\ge 2$, meanwhile, the value of $u$ is 
uniquely determined by the conditions in (b).\par

Our goal in this section is to obtain mean value estimates corresponding to auxiliary 
equations having congenial coefficient matrices. Let $D$ be an integral congenial matrix of 
type $(n,r;\rho,u,t)$. Then $D$ is an $R\times S$ matrix, where $S=3n(r-1)+t$ and 
$R=n(r-1)+\rho$. Define the linear forms $\gam_j=\gam_j(\bfalp)$ by putting
$$\gam_j(\bfalp)=\sum_{i=1}^Rd_{ij}\alp_i\quad (1\le j\le S),$$
and the Weyl sum
$$f(\alp)=\sum_{|x|\le P}e(\alp x^3).$$
Our main lemma provides an estimate for the mean value
\begin{equation}\label{2.1}
I(P;D)=\oint |f(\gam_1)\ldots f(\gam_S)|^2\d\bfalp .
\end{equation}
By considering the underlying Diophantine system, one finds that $I(P;D)$ is unchanged by 
elementary row operations on $D$, and likewise by permutations of its columns. Thus, in 
the discussion to come we may always pass to a convenient matrix row equivalent to $D$.
\par

When $\rho$, $w$, $u$ and $t$ are non-negative integers, define
\begin{equation}\label{bw1}
\del(\rho,w)=\begin{cases} 1,&\text{when $w=3\rho$,}\\
\max\{0,w-3\rho\},&\text{otherwise,}
\end{cases}
\end{equation}
and then put
\begin{equation}\label{bw2}
\del^*(\rho,u,t)=\begin{cases} \del(\rho-1,t-u)+u-3,&\text{when $\rho\ge 2$ and 
$u\ge 3$,}\\
t-2,&\text{when $\rho=1$ and $t\ge 3$,}\\
\del(\rho,t),&\text{otherwise.}
\end{cases}
\end{equation}

\begin{lemma}\label{lemma2.3}
When $\rho\ge 2$, $u\le t<3\rho$ and $u\le 2$, one has
\begin{equation}\label{bw3}
\max\{\del^*(\rho,u,t-1),\del^*(\rho-1,u,t-1)-1\}\le \del^*(\rho,u,t).
\end{equation}
Meanwhile, when $\rho\ge 2$, $t\ge 3\rho$ and $2\le u\le t$, one has
\begin{equation}\label{bw4}
\max\{\del^*(\rho,\max\{u-2,1\},t-2)+1,\del^*(\rho-1,1,t-u)+u-3\}\le \del^*(\rho,u,t).
\end{equation}
Finally, when $\rho\ge 3$ and $t\le u+\rho-1$, one has $\del^*(\rho-1,u,t)\le 
\del^*(\rho,u,t)$.
\end{lemma}

\begin{proof} We first establish (\ref{bw3}). By (\ref{bw2}), the inequality to be 
confirmed reads
\begin{equation}\label{bw5}
\max\{ \del(\rho,t-1),\del(\rho-1,t-1)-1\}\le \del(\rho,t).
\end{equation}
Since $t\le 3\rho-1$, we have $\del(\rho,t-1)=\del(\rho,t)=0$, and since 
$t-1\le 3(\rho-1)+1$, we have also $\del(\rho-1,t-1)\le 1$. The desired conclusion 
(\ref{bw5}) follows.\par

Suppose next that $\rho\ge 2$, $t\ge 3\rho$ and $t\ge u=2$. In such circumstances, the 
inequality (\ref{bw4}) to be confirmed reads
$$\max\{\del(\rho,t-2)+1,\del(\rho-1,t-2)-1\}\le \del(\rho,t).$$
By considering the cases $t\in \{ 3\rho,3\rho+1\}$, $t=3\rho+2$, and $t\ge 3\rho+3$, in 
turn, the desired conclusion follows directly from (\ref{bw1}). When instead 
$u\in \{3,4\}$, the inequality (\ref{bw4}) reads
$$\max\{\del(\rho,t-2)+1,\del(\rho-1,t-u)+u-3\}\le \del(\rho-1,t-u)+u-3,$$
and one has only to verify that $\del(\rho,t-2)\le \del(\rho-1,t-u)+u-4$. By considering the 
cases $t\in \{3\rho,3\rho+1\}$, $t=3\rho+2$, and $t\ge 3\rho+3$, in turn, the desired 
conclusion follows directly from (\ref{bw1}). Finally, when $t\ge u\ge 5$, the inequality 
(\ref{bw4}) is trivial. We have now confirmed (\ref{bw4}) in all cases.\par

In our proof of the final claim of the lemma, we may assume that $\rho\ge 3$. Thus 
$\rho+1\le 3(\rho-1)-1$ and $\rho-1\le 3(\rho-2)-1$, and hence $\del(\rho-1,\rho+1)=0$ 
and $\del(\rho-2,\rho-1)=0$. Since $t\le u+\rho-1$, it follows that when $u\le 2$, one has 
$\del(\rho-1,t)\le \del(\rho-1,\rho+1)\le \del(\rho,t)$. Likewise, when $u\ge 3$ it follows 
that
$$\del(\rho-2,t-u)\le \del(\rho-2,\rho-1)\le \del(\rho-1,t-u).$$
The desired conclusion now follows in both cases from (\ref{bw2}), completing the proof 
of the lemma.\end{proof}

For future use we record the elementary inequality
\begin{equation}\label{2.3}
|z_1\ldots z_n|\le |z_1|^n+\ldots +|z_n|^n.
\end{equation}

\begin{lemma}\label{lemma2.4}
Let $D$ be an integral congenial matrix of type $(n,r;\rho,u,t)$. Then
\begin{equation}\label{2.4}
I(P;D)\ll P^{S+\del^*(\rho,u,t)+\eps}.
\end{equation}
\end{lemma}

\begin{proof} We proceed by induction. Write $\Hhat_{n,r}^{\rho,u,t}$ to denote the 
hypothesis that the bound (\ref{2.4}) holds for all congenial matrices of type 
$(n,r;\rho,u,t)$, and $H_{n,r}^{\rho,u,t}$ to denote the hypothesis that 
$\Hhat_{n',r'}^{\rho',u',t'}$ holds for all $n'\le n$, $r'\le r$, $\rho'\le \rho$, $u'\le u$, 
$t'\le t$. Our outer induction is on $n$, with an inner induction on $r$, $\rho$, $u$ and 
$t$. The basis for this induction is provided by Hua's Lemma (see 
\cite[Lemma 2.5]{Vau1997}). This establishes that
$$\int_0^1|f(\alp)|^{2u}\d\alp \ll P^{u+\del(1,u)+\eps}\quad (u=1,2,4).$$
Thus it follows from H\"older's inequality and the trivial estimate $|f(\alp)|\le 2P+1$ that, 
when $n=0$ and $\rho=1$, then
\begin{equation}\label{2.fx}
I(P;D)=\int_0^1|f(\gam_1)\ldots f(\gam_u)|^2\d\bfalp \ll P^{u+\del(1,u)+\eps}\quad 
(u\ge 1),
\end{equation}
and one obtains $H_{0,r}^{1,u,u}$ for all $u\ge 1$. Given a congenial matrix $D$ of type 
$(0,r;\rho,u,u)$ with $\rho\ge 2$, meanwhile, one has either $u\le 2$ or $u\ge 3\rho$. It 
follows by applying elementary row operations that $D$ is row equivalent to a matrix $D'$ 
whose first row entries are all non-zero. By considering the system of equations underlying $I(P;D')$, and discarding every equation 
except that corresponding to the first row of $D'$, one finds that $I(P;D)\le I(P;D'')$, 
where $D''$ is a congenial matrix of type $(0,r;1,u,u)$. But $\del^*(\rho,u,u)\ge \del(1,u)$ 
for $u\le 2$ and also for $u\ge 6$, and thus we deduce from (\ref{2.fx}) that 
$\Hhat_{0,r}^{\rho,u,u}$ holds for all natural numbers $\rho$ and $u$.\par

Our strategy for proving the lemma involves two steps. We confirm below that when 
$\rho\ge 2$ and $t>u$, one has
\begin{equation}\label{2.5}
H_{n,r}^{\rho,u,t-1}\quad \text{implies}\quad \Hhat_{n,r}^{\rho,u,t}.
\end{equation}
Notice that when $\rho=1$, then since $\del^*(1,t,t)=\del^*(1,u,t)$, there is no loss of 
generality in supposing that $t=u$. Since $u$ (possibly zero) is the smallest value that $t$ 
can assume in a congenial matrix of type $(n,r;\rho,u,t)$, it therefore suffices to establish 
$\Hhat_{n,r}^{\rho,u,u}$ $(u\ge 1)$. We show below that when $\rho\ge 1$, then
\begin{equation}\label{rv2.12}
(\text{$\Hhat_{n-1,r}^{r,1,3(r-1)}$ and $\Hhat_{n,r}^{1,u,u}$})\quad 
\text{implies}\quad \Hhat_{n,r}^{\rho,u,u}.
\end{equation}
Since there is no loss in supposing that a congenial matrix of type $(n,r;1,u,u)$ is also of 
type $(n-1,r;r,\max\{u,1\},3(r-1)+u)$, one finds via (\ref{rv2.12}) that 
$H_{n-1,r}^{r,u+1,3r+u}$ implies $\Hhat_{n,r}^{1,u,u}$, and hence also 
$\Hhat_{n,r}^{\rho,u,u}$. We note in this context that 
$\del^*(1,u,u)=\del^*(r,\max\{u,1\},3(r-1)+u)$. In view of (\ref{2.5}), one sees that 
whenever $H_{n-1,r}^{\sig,v,v}$ holds for all $\sig$ and $v$, then one has 
$\Hhat_{n,r}^{\rho,u,u}$ for all $\rho$ and $u$, and hence also $H_{n,r}^{\rho,u,t}$ for 
all $\rho$, $u$ and $t$. We have already established $\Hhat_{0,r}^{\sig,v,v}$ for all 
$\sig$ and $v$, and hence the conclusion of the lemma follows by induction on $n$.\par

We begin by confirming (\ref{rv2.12}). Let $D$ be congenial of type $(n,r;\rho,u,u)$, and 
suppose $\Hhat_{n-1,r}^{r,1,3(r-1)}$ and $\Hhat_{n,r}^{1,u,u}$. We may suppose that 
$\rho\ge 2$, for otherwise (\ref{rv2.12}) is 
trivial. Since the first $u$ columns of $D$ define a subspace of dimension $1$ distinct from 
the $\rho$-th coordinate axis, the matrix $D$ has non-zero entries populating one of its 
first $\rho-1$ rows in the first $u$ columns. The matrix $D$ is consequently row 
equivalent to one of separated block form, with one block $D_0$ of format $1\times u$ 
(trivially) congenial of type $(0,r;1,u,u)$, and the second block $D_1$ of format 
$(R-\rho+1)\times (S-u)$. There is no loss of generality in supposing $D_1$ to be 
congenial of type $(n-1,r;r,1,3(r-1))$. On considering the underlying Diophantine systems, 
we therefore find that $I(P;D)=I(P;D_0)I(P;D_1)$. We may assume (\ref{2.fx}) and 
$\Hhat_{n-1,r}^{r,1,3(r-1)}$. Thus we deduce via (\ref{bw1}) and (\ref{bw2}) that
\begin{align*}
I(P;D)&\ll P^{u+\del(1,u)+\eps}\cdot P^{S-u+\del^*(r,1,3(r-1))+\eps}\\
&= P^{S+\del(1,u)+\del(r,3(r-1))+2\eps}=P^{S+\del(1,u)+2\eps}.
\end{align*}
When $\rho\ge 2$, the congeniality of $D$ ensures that either $u\le 2$ or $u\ge 6$, and 
hence $\del^*(\rho,u,u)=\del(1,u)$. Thus $I(P;D)\ll P^{S+\del^*(\rho,u,u)+\eps}$, 
confirming (\ref{rv2.12}).\par

We now commence the proof of (\ref{2.5}). Let $D$ be a matrix of type $(n,r;\rho,u,t)$ 
with $\rho\ge 2$ and $u<t$, and suppose $H_{n,r}^{\rho,u,t-1}$. Should the first 
$\rho-1$ rows of the matrix $D$ be linearly dependent, then by applying elementary row 
operations on these rows, we may suppose that $D$ is congenial with one of these rows 
zero. Thus $t-u+1<\rho$ and $\rho\ge 3$, and on deleting this row and applying the 
final conclusion of Lemma \ref{lemma2.3}, it is apparent that (\ref{2.4}) will be confirmed 
provided that we establish the bound $I(P;D)\ll P^{S+\del^*(\rho-1,u,t)+\eps}$. 
Repeated use of this simplification permits us to condition the first $\rho-1$ rows of $D$ 
to be linearly independent. We divide into cases according to whether $t<3\rho$ or 
$t\ge 3\rho$.\par

We first establish (\ref{2.5}) in the situation where $t<3\rho$. One then has $u\le 2$. We 
distinguish three cases. When $t=u+1$, it follows from the conditioned congeniality of $D$ 
that $\rho=2$ or $3$. In such circumstances, we say that $D$ has type I when 
$\gam_t=d_{\rho,t}\alp_\rho$ with $d_{\rho,t}\ne 0$. Note that the conditioned 
congeniality of $D$ then implies that $\rho=2$. When $t=u+1$ and $D$ is not of type I, 
we apply elementary row operations to ensure that $\gam_t=d_{1,t}\alp_1$ with 
$d_{1,t}\ne 0$, and also that $d_{2,j}\ne 0$ for $1\le j\le u$. We describe the resulting 
matrix as having type~II. A conditioned congenial matrix $D$ not of type I or II we 
describe as having type III. For such matrices, one has $\rho\ge 3$ and $t\ge u+2$.\par

Consider first a matrix $D$ of type I. By performing elementary row operations, one may 
suppose that $\gam_j=d_{1,j}\alp_1$, with $d_{1,j}\ne 0$, for $1\le j\le u$. The matrix 
$D$ is of separated block form, with one block $D_0$ of format $1\times u$ (trivially) 
congenial of type $(0,1;1,u,u)$, and the second block $D_1$ of format $(R-1)\times (S-u)$ 
congenial of type $(n,r;1,1,1)$. On considering the underlying Diophantine systems, we 
find that $I(P;D)\ll I(P;D_0)I(P;D_1)$. We may assume (\ref{2.fx}) and 
$\Hhat_{n,r}^{1,1,1}$, and thus
$$I(P;D)\ll P^{u+\del(1,u)+\eps}\cdot P^{S-u+\del^*(1,1,1)+\eps}=
P^{S+2\eps}\ll P^{S+\del^*(2,u,u+1)+\eps},$$ 
confirming the estimate (\ref{2.4}) in this case.\par

Next consider a matrix $D$ of type III. The last $t-u+1$ columns of $A_0$ span a linear 
space of dimension $\min\{t-u+1,\rho\}\ge 3$. Hence, by permuting the last $t-u$ 
columns of $A_0$, we may suppose that the $t$-th column of $A_0$ is not contained in 
the linear space generated by the first $u$ columns and the $\rho$-th coordinate vector. 
By applying elementary row operations, we may arrange that the conditioned matrix $D$ 
is congenial with $\gam_t=d_{1,t}\alp_1$ and $d_{1,t}\ne 0$. We note that the linear 
space spanned by the first $u$ columns of $D$ is now distinct from both the first and the 
$\rho$-th coordinate axis. Let $D_0$ denote the matrix obtained from $D$ by deleting 
column $t$, and let $D_1$ denote the matrix obtained by instead deleting row $1$ and 
column $t$. Lemma \ref{lemma2.1} shows the $R\times (S-1)$ matrix $D_0$ to be 
congenial of type $(n,r;\rho,u,t-1)$, and the $(R-1)\times (S-1)$ matrix $D_1$ to be 
congenial of type $(n,r;\rho-1,u,t-1)$. Observe that when $D$ is a matrix of type II, then 
$\rho=2$ or $3$, and this same conclusion holds.\par

We may now consider matrices $D$ of types II and III together. We have 
$\gam_t=d_{1,t}\alp_1$ with $d_{1,t}\ne 0$. 
Weyl differencing (see \cite[equation (2.6)]{Vau1997}) yields
$$|f(\gam_t)|^2\ll P+\sum_{0<|h|\le 16P^3}c_he(\gam_th),$$
where the integers $c_h$ satisfy $c_h=O(|h|^\eps)$. We therefore find from (\ref{2.1}) 
that
\begin{equation}\label{2.7}
I(P;D)\ll PT(0)+\sum_{0<|h|\le 16P^3}c_hT(h),
\end{equation}
where
\begin{equation}\label{2.8}
T(h)=\oint \prod_{\substack{1\le i\le S\\ i\ne t}}|f(\gam_i)|^2e(\gam_th)\d\bfalp .
\end{equation}
The contribution of the terms with $h\ne 0$ in (\ref{2.7}) is given by
\begin{equation}\label{2.9}
\sum_{0<|h|\le 16P^3}c_hT(h)\ll P^\eps \oint \prod_{\substack{1\le i\le S\\ i\ne t}}|
f({\widehat \gam}_i)|^2\d{\widehat \bfalp},
\end{equation}
where ${\widehat \bfalp}=(\alp_2,\ldots ,\alp_R)$ and 
${\widehat \gam}_m=\gam_m(0,\alp_2,\ldots ,\alp_R)$. On considering the underlying 
Diophantine systems, we discern on the one hand from (\ref{2.8}) that $T(0)=I(P;D_0)$, 
and on the other that the integral on the right hand side of (\ref{2.9}) is equal to 
$I(P;D_1)$. Thus
$$I(P;D)\ll PI(P;D_0)+P^\eps I(P;D_1).$$
We may assume $H_{n,r}^{\rho,u,t-1}$, and thus Lemma \ref{lemma2.3} yields the 
estimate
$$I(P;D)\ll P^{S+\del^*(\rho,u,t-1)+\eps}+P^{S-1+\del^*(\rho-1,u,t-1)+2\eps}
\ll P^{S+\del^*(\rho,u,t)+2\eps}.$$
This confirms the bound (\ref{2.4}), and hence (\ref{2.5}) holds whenever $t<3\rho$.\par

We turn next to the situation with $t\ge 3\rho$. Recall that $\rho\ge 2$ and $u\ge 1$. By 
relabelling the first $t$ columns of $D$, we may assume without loss that the conditioned 
congenial matrix $D$ has the property that, should any one of these columns lie on the 
$\rho$-th coordinate axis, then this is the $t$-th column. Then, applying the bound 
(\ref{2.3}) within (\ref{2.1}), one finds that with $j=1$ or $2$, one has
\begin{equation}\label{2.xy}
I(P;D)\ll \oint|f(\gam_j)^4f(\gam_3)^2\ldots f(\gam_S)^2|\d\bfalp .
\end{equation}
Thus, by symmetry, we may suppose that $j=1$ and $u\ge 2$. Since the first $u$ columns 
of $D$ lie in a subspace of dimension $1$ distinct from the $\rho$-th coordinate axis, by 
applying elementary row operations, we see that there is no loss of generality in assuming 
that the congenial matrix $D$ satisfies the condition that $\gam_j=d_{1,j}\alp_1$ with 
$d_{1,j}\ne 0$ for $1\le j\le u$.\par

We first examine the situation in which $\rho\ge 2$, $2\le u\le 4$ and $t\ge 3\rho\ge 6$. 
By Weyl differencing (see \cite[equation (2.6)]{Vau1997}), one has
$$|f(\gam_1)|^4\ll P^3+P\sum_{0<|h|\le 32P^3}b_he(\gam_1h),$$
where the integers $b_h$ satisfy $b_h=O(|h|^\eps)$. We therefore find from (\ref{2.xy}) 
that
\begin{equation}\label{2.11}
I(P;D)\ll P^3U(0)+P\sum_{0<|h|\le 32P^3}b_hU(h),
\end{equation}
where
\begin{equation}\label{2.12}
U(h)=\oint |f(\gam_3)\ldots f(\gam_S)|^2e(\gam_1h)\d\bfalp .
\end{equation}
The contribution of the terms with $h\ne 0$ in (\ref{2.11}) is given by
\begin{equation}\label{2.13}
P\sum_{0<|h|\le 32P^3}b_hU(h)\ll P^{1+\eps}\oint |f({\widehat \gam}_3)\ldots 
f({\widehat \gam}_S)|^2\d{\widehat \bfalp }.
\end{equation}
Let $D_0$ now denote the matrix obtained from $D$ by deleting the first two columns, 
and let $D_1$ denote the matrix obtained by instead deleting row $1$ and the first $u$ 
columns. Since $t\ge 6$, Lemma \ref{lemma2.1} shows the $R\times (S-2)$ matrix $D_0$ 
to be congenial of type $(n,r;\rho,\max\{u-2,1\},t-2)$, and the $(R-1)\times (S-u)$ matrix 
$D_1$ to be congenial of type $(n,r;\rho-1,1,t-u)$. On considering the underlying 
Diophantine systems, we find on the one hand from (\ref{2.12}) that $U(0)=I(P;D_0)$, 
and on the other that the integral on the right hand side of (\ref{2.13}) is equal to 
$P^{2u-4}I(P;D_1)$. Here, we have made use of the fact that 
${\widehat \gam}_m(\bfalp)=0$ for $3\le m\le u$. Thus
$$I(P;D)\ll P^3I(P;D_0)+P^{2u-3+\eps}I(P;D_1).$$
We may assume $H_{n,r}^{\rho,u,t-1}$, and thus Lemma \ref{lemma2.3} delivers the 
estimate
\begin{align*}
I(P;D)&\ll P^{S+1+\del^*(\rho,\max\{u-2,1\},t-2)+\eps}+
P^{S+u-3+\del^*(\rho-1,1,t-u)+2\eps}\\
&\ll P^{S+\del^*(\rho,u,t)+2\eps}.
\end{align*}
Since $\del^*(\rho,1,t)=\del^*(\rho,2,t)$, we obtain (\ref{2.4}) even when the case 
$u=1$ was simplified to that with $u=2$.\par

Finally, suppose that $u\ge 5$. Recall that $\gam_j=d_{1,j}\alp_1$, with $d_{1,j}\ne 0$, 
for $1\le j\le u$. Let $D_0$ denote the matrix obtained from $D$ by deleting all but the 
first row and all but the first $u$ columns, and let $D_1$ denote the matrix obtained by 
instead deleting the first row and first $u$ columns. Then the $1\times u$ matrix $D_0$ is 
(trivially) congenial of type $(0,1;1,u,u)$, and the $(R-1)\times (S-u)$ matrix $D_1$ is 
congenial of type $(n,r;\rho-1,1,t-u)$. On considering the underlying Diophantine systems 
and applying the triangle inequality, we find via (\ref{2.fx}) that
$$I(P;D)\ll I(P;D_0)I(P;D_1)\ll P^{2u-3+\eps}I(P;D_1),$$
which as above confirms the estimate (\ref{2.4}) in this final case. Hence we have 
completed the proof of (\ref{2.5}) when $t\ge 3\rho$, completing the proof of the 
inductive step. The conclusion of the lemma now follows.
\end{proof}

We extract a simple consequence from this lemma for future use.

\begin{corollary}\label{corollary2.5}
Let $r\ge 2$, suppose that $D$ is an integral congenial matrix of type $(n,r;r,3,3r)$, and 
write $w=(n+1)r-n$. Then $I(P;D)\ll P^{3w+1+\eps}$.
\end{corollary}

\begin{proof} We have only to note that $\del^*(r,3,3r)=\del(r-1,3r-3)=1$.
\end{proof}

\section{Complification} Before describing the process which leads from the basic mean value to the more
 complicated ones described in the previous section, we introduce some additional Weyl sums. When 
$2\le R\le P$, we put
$$\calA(P,R)=\{n\in [-P,P]\cap \dbZ: \text{$p$ prime and $p|n$} \Rightarrow p\le R\},$$
and then define the exponential sum $g(\alp)=g(\alp;P,R)$ by
$$g(\alp;P,R)=\sum_{x\in \calA(P,R)}e(\alp x^3).$$
We find it convenient to write $\tau$ for any positive number satisfying $\tau^{-1}>852+16\sqrt{2833}=1703.6\ldots $, and then put $\xi=\tfrac{1}{4}-\tau$.

\begin{lemma}\label{lemma3.1} When $\eta$ is sufficiently small and $2\le R\le P^\eta$, one has
$$\int_0^1|g(\alp;P,R)|^6\d\alp \ll P^{3+\xi}.$$
\end{lemma}

\begin{proof} The conclusion follows from \cite[Theorem 1.2]{Woo2000} by considering the underlying
 Diophantine equations.\end{proof}

Next we establish an auxiliary lemma that executes the complification process. Let $n$ and $r$ be 
non-negative integers with $r\ge 2$, and write $R=n(r-1)$ and $S=3R$. Let $B=(b_{i,j})$ be an integral
 $(R+1)\times (S+2)$ matrix, write $\bfb_j$ for the column vector $(b_{i,j})_{1\le i\le R+1}$, and define 
$\bfb_j^*$ to be the column vector $(b_{R+2-i,j})_{1\le i\le R+1}$ in which the entries of 
$\bfb_j$ are flipped upside down. Also, define $\bet_j=\bet_j(\bfalp)$ by putting
\begin{equation}\label{3.1}
\bet_j(\bfalp)=\sum_{i=1}^{R+1}b_{ij}\alp_i\quad (0\le j\le S+1).
\end{equation}
We say that the matrix $B$ is {\it bicongenial of type} $(n,r)$ when (i) the column vectors 
$\bfb_0,\bfb_1,\ldots ,\bfb_S$ and $\bfb^*_{S+1},\bfb^*_S,\ldots ,\bfb^*_1$ both form 
congenial matrices having type $(n-1,r;r,1,3r-2)$, and (ii) one has 
$\bet_0(\bfalp)=b_{1,0}\alp_1$ and $\bet_{S+1}(\bfalp)=b_{R+1,S+1}\alp_{R+1}$. At 
this point, we introduce the mean value
\begin{equation}\label{3.2}
J(P;B)=\oint |g(\bet_0)^3f(\bet_1)^2\ldots f(\bet_S)^2g(\bet_{S+1})^3|\d\bfalp .
\end{equation}
Finally, we fix $\eta>0$ to be sufficiently small in the context of Lemma \ref{lemma3.1}.

\begin{lemma}\label{lemma3.2} Suppose that $B$ is an integral bicongenial matrix of type 
$(n,r)$. Then there exists an integral bicongenial matrix $B^*$ of type $(2n,r)$ for which
$$J(P;B)\ll (P^{3+\xi})^{1/2}J(P;B^*)^{1/2}.$$
\end{lemma}

\begin{proof} Define the linear forms $\bet_j$ as in (\ref{3.1}). Also, define 
$$T(P;B)=\int_0^1\left( \oint |g(\bet_0)^3f(\bet_1)^2\ldots f(\bet_S)^2|
\d{\widehat \bfalp}_R\right)^2\d\alp_{R+1},$$
where ${\rm d}{\widehat \bfalp}_R$ denotes ${\rm d}\alp_1\ldots \d\alp_R$. Then Schwarz's inequality
 leads from (\ref{3.2}) to the bound
\begin{equation}\label{3.3}
J(P;B)\le \Bigl( \int_0^1|g(\bet_{S+1})|^6\d\alp_{R+1}\Bigr)^{1/2}T(P;B)^{1/2}.
\end{equation}
By expanding the square inside the outermost integration, we see that
$$T(P;B)=\oint |g(\bet_0^*)^3f(\bet_1^*)^2\ldots f(\bet_{2S}^*)^2g(\bet_{2S+1}^*)^3|
\d{\widehat \bfalp}_{2R+1},$$
where $\bet_i^*=\bet_i^*(\bfalp)$ is defined by
$$\bet_i^*(\bfalp)=\begin{cases}
\bet_i(\alp_1,\ldots ,\alp_{R+1}),&\text{when $0\le i\le S$,}\\
\bet_{2S+1-i}(\alp_{2R+1},\ldots ,\alp_{R+1}),&\text{when $S+1\le i\le 2S+1$.}
\end{cases}$$
The integral $(2R+1)\times (2S+2)$ matrix $B^*=(b_{ij}^*)$ defining the linear forms 
$\bet_0^*,\ldots ,\bet_{2S+1}^*$ is bicongenial of type $(2n,r)$, and one has $T(P;B)=J(P;B^*)$. The
 conclusion of the lemma therefore follows from (\ref{3.3}) and Lemma \ref{lemma3.1}.
\end{proof}

While Lemma \ref{lemma3.2} bounds $J(P;B)$ in terms of a mean value almost twice the original
 dimension, superficially {\it complicating} the task at hand, the higher dimension in fact {\it simplifies} the
 problem of obtaining close to square root cancellation. Hence our use of the term {\it complification}.\par

Consider an $r\times s$ integral matrix $C=(c_{ij})$, write $\bfc_j$ for the column vector
 $(c_{ij})_{1\le i\le r}$, and put
\begin{equation}\label{3.4}
\gam_j=\sum_{i=1}^rc_{ij}\alp_i\quad (1\le j\le s).
\end{equation}
Also, when $s\ge 3$, write
\begin{equation}\label{3.5}
K(P;C)=\oint |g(\gam_1)g(\gam_2)g(\gam_3)f(\gam_4)\ldots f(\gam_s)|^2\d\bfalp .
\end{equation}

\begin{theorem}\label{theorem3.3}
Suppose that $r\ge 2$ and that the $r\times 3r$ integral matrix $C$ is highly non-singular. Then 
$K(P;C)\ll P^{3r+\xi+\eps}$.
\end{theorem}

\begin{proof} Write $s=3r$. Since the $r\times s$ matrix $C$ is highly non-singular with 
$r\ge 2$, we may apply elementary row operations to $C$ in such a manner that 
$c_{1,1}\ne 0$, $c_{r,2}\ne 0$, and $c_{1,3}c_{r,3}\ne 0$. On considering the underlying 
Diophantine system, it is apparent from (\ref{3.5}) that these operations leave the mean value 
$K(P;C)$ unchanged. Next, by applying the elementary relation (\ref{2.3}) within 
(\ref{3.5}), one finds by symmetry that there is no loss in supposing that
$$K(P;C)\ll \oint |g(\gam_1)^3f(\gam_4)^2\ldots f(\gam_s)^2g(\gam_2)^3|\d\bfalp .$$
By relabelling the linear forms, we infer that $K(P;C)\ll J(P;B_0)$, where $B_0$ is the matrix 
with columns $\bfc_1,\bfc_4,\bfc_5,\ldots ,\bfc_{s-1},\bfc_s,\bfc_2$. From here, by applying 
elementary row operations, which amounts to making a non-singular change of variable 
within (\ref{3.5}), we may suppose that $\gam_1=c_{1,1}\alp_1$ and 
$\gam_2=c_{r,2}\alp_r$. Since the $r\times s$ matrix $C$ is highly non-singular, Lemma 
\ref{lemma2.1} shows that $B_0$ is bicongenial of type $(1,r)$.\par

We show by induction that for each non-negative integer $l$, there exists an integral 
bicongenial matrix of type $(2^l,r)$ having the property that
\begin{equation}\label{3.6}
K(P;C)\ll (P^{3+\xi})^{1-2^{-l}}J(P;B_l)^{2^{-l}}.
\end{equation}
This bound holds when $l=0$ as a trivial consequence of the upper bound $K(P;C)\ll J(P;B_0)$ just
 established. Suppose then that the estimate (\ref{3.6}) holds for $0\le l\le L$. By applying Lemma
 \ref{lemma3.2}, we see that there exists an integral bicongenial matrix $B_{L+1}$ of type $(2^{L+1},r)$
 having the property that
$$J(P;B_L)\ll (P^{3+\xi})^{1/2}J(P;B_{L+1})^{1/2}.$$
Substituting this estimate into the case $l=L$ of (\ref{3.6}), one confirms that (\ref{3.6}) holds with
 $l=L+1$. The bound (\ref{3.6}) therefore follows for all $l$ by induction.\par

We now prepare to apply the bound just established. Let $\del$ be any small positive number, and choose
 $l$ large enough that $2^{1-l}(1-\xi)<\del$. We have shown that an integral bicongenial matrix
 $B_l=(b_{ij})$ exists for which (\ref{3.6}) holds. The matrix $B_l$ is of format $(R+1)\times (S+2)$,
 where $R=2^l(r-1)$ and $S=3R$. Define the linear forms $\bet_j$ as in (\ref{3.1}) and recall (\ref{3.2}).
 Applying (\ref{2.3}), invoking symmetry, and considering the underlying Diophantine system, we find that
 there is no loss in supposing that
$$J(P;B_l)\ll \oint |f(\bet_0)^6f(\bet_1)^2\ldots f(\bet_S)^2|\d\bfalp .$$
Let $D$ be the integral matrix underlying the $S+3$ forms $\bet_0,\bet_0,\bet_0,\bet_1,\ldots ,\bet_S$. Then $D$ is congenial of
 type $(2^l-1,r;r,3,3r)$, and one has $J(P;B_l)\ll I(P;D)$. Substituting the bound $J(P;B_l)\ll
 P^{3R+4+\eps}$ that follows from Corollary \ref{corollary2.5} into (\ref{3.6}), we obtain the estimate
$$K(P;C)\ll (P^{3+\xi})^{1-2^{-l}}(P^{3(2^l(r-1)+1)+1+\eps})^{2^{-l}}\ll 
P^{3r+\xi+(1-\xi)2^{-l}+\eps}.$$
In view of our assumed upper bound $2^{1-l}(1-\xi)<\del$, one therefore sees that
$$K(P;C)\ll P^{3r+\xi+\frac{1}{2}\del+\eps}\ll P^{3r+\xi+\del}.$$
The conclusion of the theorem now follows by taking $\del$ sufficiently small.
\end{proof}

\section{The Hardy-Littlewood method} In this section we turn to the proof of Theorem \ref{theorem1.1}.
 Let $(c_{ij})$ denote an integral $r\times s$ highly non-singular matrix with $r\ge 2$ and $s\ge 6r+1$. 
We define the linear forms $\gam_j=\gam_j(\bfalp)$ as in (\ref{3.4}), and for concision put
 $g_j=g(\gam_j(\bfalp))$ and $f_j=f(\gam_j(\bfalp))$. When $\grB\subseteq [0,1)^r$ is measurable, we
 then define
$$N(P;\grB)=\int_\grB g_1\ldots g_6f_7\ldots f_s\d\bfalp .$$
By orthogonality, it follows from this definition that $N(P;[0,1)^r)$ counts the number of integral solutions
 of the system (\ref{1.1}) with $x_1,\ldots ,x_6\in \calA(P,R)$ and $x_7,\ldots ,x_s\in [-P,P]$. In this section
 we prove the lower bound $N(P;[0,1)^r)\gg P^{s-3r}$, subject to the hypothesis that the system
 (\ref{1.1}) has non-zero $p$-adic solutions for all primes $p$. The conclusion of Theorem
 \ref{theorem1.1} then follows.\par

In pursuit of the above objective, we apply the Hardy-Littlewood method. Let $\grM$ denote the union of
 the intervals
$$\grM(q,a)=\{ \alp\in [0,1):|q\alp -a|\le (6P^2)^{-1}\},$$
with $0\le a\le q\le P$ and $(a,q)=1$, and let $\grm=[0,1)\setminus \grM$. In addition, write 
$L=\log \log P$, denote by $\grN$ the union of the intervals
$$\grN(q,a)=\{\alp \in [0,1):|q\alp -a|\le LP^{-3}\},$$
with $0\le a\le q\le L$ and $(a,q)=1$, and put $\grn=[0,1)\setminus \grN$. We summarise some useful
 estimates in this context in the form of a lemma.

\begin{lemma}\label{lemma4.1}
One has
$$\int_{\grM\setminus \grN}|f(\alp)|^5\d\alp \ll P^2L^{\eps-1/3}\quad \text{and}\quad 
\int_0^1|f(\alp)|^8\d\alp \ll P^5.$$
\end{lemma}

\begin{proof} The first estimate follows as a special case of \cite[Lemma 5.1]{Vau1989}, and the second is
 immediate from \cite[Theorem 2]{Vau1986}, by orthogonality.
\end{proof}

Next we introduce a multi-dimensional set of arcs. Let $Q=L^{10r}$, and define the narrow set of major
 arcs $\grP$ to be the union of the boxes
$$\grP(q,\bfa)=\{ \bfalp\in [0,1)^r:|\alp_i-a_i/q|\le QP^{-3}\ (1\le i\le r)\},$$
with $0\le a_i\le q\le Q$ $(1\le i\le r)$ and $(a_1,\ldots ,a_r,q)=1$.

\begin{lemma}\label{lemma4.2} Suppose that the system (\ref{1.1}) admits non-zero $p$-adic solutions
 for each prime number $p$. Then one has $N(P;\grP)\gg P^{s-3r}$.
\end{lemma}

\begin{proof}We begin by defining the auxiliary functions
$$S(q,a)=\sum_{r=1}^qe(ar^3/q)\quad \text{and}\quad v(\bet)=\int_{-P}^Pe(\bet \gam^3)\d\gam .$$
For $1\le j\le s$, put $S_j(q,\bfa)=S(q,\gam_j(\bfa))$ and $v_j(\bfbet)=v(\gam_j(\bfbet))$, and define
\begin{equation}\label{4.1}
A(q)=\underset{(q,a_1,\ldots ,a_r)=1}{\sum_{a_1=1}^q\cdots \sum_{a_r=1}^q}
q^{-s}\prod_{j=1}^sS_j(q,\bfa)\quad \text{and}\quad V(\bfbet)=\prod_{j=1}^sv_j(\bfbet).
\end{equation}
Finally, write $\calB(X)$ for $[-XP^{-3},XP^{-3}]^r$, and define
$$\grJ(X)=\int_{\calB(X)}V(\bfbet)\d\bfbet \quad \text{and}\quad \grS(X)=\sum_{1\le q\le X}A(q).$$

\par We prove first that there exists a positive constant $C$ with the property that
\begin{equation}\label{4.2}
N(P;\grP)-C\grS(Q)\grJ(Q)\ll P^{s-3r}L^{-1}.
\end{equation}
It follows from \cite[Lemma 8.5]{Woo1991} (see also \cite[Lemma 5.4]{Vau1989}) that there exists a
 positive constant $c=c(\eta)$ such that whenever $\bfalp \in \grP(q,\bfa)\subseteq \grP$, then
$$g(\gam_j(\bfalp))-cq^{-1}S_j(q,\bfa)v_j(\bfalp-\bfa/q)\ll P(\log P)^{-1/2}.$$
Under the same constraints on $\bfalp$, one finds from \cite[Theorem 4.1]{Vau1997} that
$$f(\gam_j(\bfalp))-q^{-1}S_j(q,\bfa)v_j(\bfalp-\bfa/q)\ll \log P.$$ 
Thus, whenever $\bfalp\in \grP(q,\bfa)\subseteq \grP$, one has
$$g_1\ldots g_6f_7\ldots f_s-c^6q^{-s}\prod_{j=1}^sS_j(q,\bfa)v_j(\bfalp-\bfa/q)\ll 
P^s(\log P)^{-1/2}.$$
The measure of the major arcs $\grP$ is $O(Q^{2r+1}P^{-3r})$, so that on integrating over $\grP$, we 
confirm the relation (\ref{4.2}) with $C=c^6$.\par

We next discuss the singular integral $\grJ(Q)$. By applying (\ref{2.3}), we find that
\begin{equation}\label{4.2bw}
V(\bfbet)\ll \sum_{1\le j_1<\ldots <j_r\le s}|v_{j_1}(\bfbet)\ldots v_{j_r}(\bfbet)|^{s/r}.
\end{equation}
Recall from \cite[Theorem 7.3]{Vau1997} that $v(\bet)\ll P(1+P^3|\bet|)^{-1/3}$. Since $(c_{ij})$ is
 highly non-singular and $s\ge 6r+1$, a change of variables reveals that $V(\bfbet)$ is integrable, that the 
limit $\grJ=\underset{X\rightarrow \infty}\lim \grJ(X)$ exists, and that $\grJ\ll P^{s-3r}$. Write 
${\widehat \calB}(X)=\dbR^r\setminus \calB(X)$. Then by applying (\ref{4.2bw}), we discern that there are 
distinct indices $j_1,\ldots ,j_r$ such that
$$\grJ-\grJ(X)=\int_{{\widehat \calB}(X)}V(\bfbet)\d\bfbet \ll \int_{{\widehat \calB}(X)}|v_{j_1}
(\bfbet)\ldots v_{j_r}(\bfbet)|^{s/r}\d\bfbet .$$
The linear independence of the $\gam_j$ ensures that whenever $\bfbet \in {\widehat \calB}(X)$, then 
for some index $l$ with $1\le l\le r$, one has $|\gam_{j_l}(\bfbet)|>X^{1/2}P^{-3}$. Consequently, the
 hypothesis $s\ge 6r+1$ again ensures via a change of variables that
\begin{align*}
\grJ-\grJ(X)&\ll \Bigl( \sup_{\bfbet\in {\widehat \calB}(X)}|v_{j_1}(\bfbet)\ldots v_{j_r}(\bfbet)|\Bigr)
 \int_{{\widehat \calB}(X)}|v_{j_1}(\bfbet)\ldots v_{j_r}(\bfbet)|^{(s-r)/r}\d\bfbet \\
&\ll P^sX^{-1/6}\int_{\dbR^r}\prod_{j=1}^r(1+P^3|\tet_i|)^{-(s-r)/(3r)}\d\bftet \ll P^{s-3r}X^{-1/6}.
\end{align*}
The system of equations (\ref{1.1}) possesses a non-zero real solution in $[-1,1]^s$, and this must be 
non-singular since $(c_{ij})$ is highly non-singular. An application of Fourier's integral formula (see
 \cite[Chapter 4]{Dav2005} and \cite[Lemma 30]{DL1969}) therefore leads to the lower bound 
$\grJ\gg P^{s-3r}$. Thus we may conclude that
\begin{equation}\label{4.3}
\grJ(Q)\gg P^{s-3r}+O(P^{s-3r}Q^{-1/6})\gg P^{s-3r}.
\end{equation}

\par We turn next to the singular series $\grS(Q)$. It follows from \cite[Theorem 4.2]{Vau1997} that
 whenever $(q,a)=1$, one has $S(q,a)\ll q^{2/3}$. Given a summand $\bfa$ in the formula for $A(q)$
 provided in (\ref{4.1}), write $h_j=(q,\gam_j(\bfa))$. Then we find that
$$A(q)\ll \underset{(q,a_1,\ldots ,a_r)=1}{\sum_{a_1=1}^q\cdots \sum_{a_r=1}^q}q^{-s/3}
(h_1\ldots h_s)^{1/3}.$$
By hypothesis, we have $s/(3r)\ge 2+1/(3r)$. The proof of \cite[Lemma 23]{DL1969} is therefore easily
 modified to show that $A(q)\ll q^{-1-1/(6r)}$. Thus, the series $\grS=\underset{X\rightarrow \infty}\lim
 \grS(X)$ is absolutely convergent and
$$\grS-\grS(Q)\ll \sum_{q>Q}q^{-1-1/(6r)}\ll Q^{-1/(6r)}\ll L^{-1}.$$
The system (\ref{1.1}) has non-zero $p$-adic solutions for each prime $p$, and these are non-singular 
 since $(c_{ij})$ is highly non-singular. A modification of the proof of \cite[Lemma 31]{DL1969} therefore
 shows that $\grS>0$, whence $\grS(Q)=\grS+O(L^{-1})>0$. The proof of the lemma is completed by
 recalling (\ref{4.3}) and substituting into (\ref{4.2}) to obtain the bound $N(P;\grP)\gg
 P^{s-3r}+O(P^{s-3r}L^{-1})$.
\end{proof}

In order to prune a wide set of major arcs down to the narrow set $\grP$ just considered, we introduce the
 auxiliary sets of arcs
$$\grM_j=\{\bfalp \in [0,1)^r:\gam_j(\bfalp)\in \grM+\dbZ\},$$
and we put $\grV=\grM_7\cap\grM_{8}\cap\ldots \cap \grM_s$. In addition, we define
 $\grm_j=[0,1)^r\setminus \grM_j$ $(7\le j\le s)$, and write $\grv=[0,1)^r\setminus \grV$. Finally, for any 
positive integer $n$, when
 $\bfome\in [1,s]^n$, we define
$$\grK_\bfome=\{\bfalp\in \grV\setminus \grP:\gam_{\ome_m}(\bfalp)\in \grn+\dbZ\ (1\le m\le n)\}.
$$

\begin{lemma}\label{lemma4.3} One has $N(P;\grV\setminus \grP)\ll P^{s-3r}L^{-1/4}$.
\end{lemma}

\begin{proof} Let $\bfalp\in \grV\setminus \grP$, and suppose temporarily that $\gam_{j_m}\in
 \grN+\dbZ$ for $r$ distinct indices $j_m\in [7,s]$. For each $m$ there is a natural number $q_m\le L$
 having the property that $\|q_m\gam_{j_m}\|\le LP^{-3}$. With $q=q_1\ldots q_r$, one has $q\le L^r$
 and $\|q\gam_{j_m}\|\le L^rP^{-3}$. Next eliminating between $\gam_{j_1},\ldots ,\gam_{j_r}$ in
 order to isolate $\alp_1,\ldots ,\alp_r$, one finds that there is a positive integer $\kap$, depending at most
 on $(c_{ij})$, such that $\|\kap q\alp_l\|\le L^{r+1}P^{-3}$ $(1\le l\le r)$. Since $\kap q\le L^{r+1}$, it
 follows that $\bfalp\in \grP$, yielding a contradiction to our hypothesis that $\bfalp\in \grV\setminus \grP$.
 Thus $\gam_\nu(\bfalp)\in \grn+\dbZ$ for at least $s-6-r\ge 5(r-1)$ of the suffices $\nu$ with $7\le \nu
 \le s$. Then for some tuple $\bfnu=(\nu_1,\ldots ,\nu_{5r-5})$ of distinct integers $\nu_m\in [7,s]$, one
 has
$$N(P;\grV\setminus \grP)\ll \int_{\grK_\bfnu}|g_1\ldots g_6f_7\ldots f_s|\d\bfalp .$$

\par By symmetry, we may suppose that $\bfnu=(9,\ldots ,5r+3)$. Let $k_l$ denote $g_l$ when 
$1\le l\le 6$, and $f_l$ when $l=7,8$. Then combining (\ref{2.3}) with a trivial estimate for $|f(\alp)|$,
 one finds that for some tuple $(\sig_1,\ldots ,\sig_{r-1})$ of distinct integers $\sig_m\in [9,5r+3]$, and
 some integer $l$ with $1\le l\le 8$, one has
$$N(P;\grV\setminus \grP)\ll P^{s-5r-3}\int_{\grK_\bfsig}|k_l^8f_{\sig_1}^5
\ldots f_{\sig_{r-1}}^5|\d\bfalp .$$
By changing variables, considering the underlying Diophantine equations, and applying Lemma
 \ref{lemma4.1}, we deduce that
\begin{align*}
N(P;\grV\setminus \grP)&\ll P^{s-5r-3}\Bigl( \int_0^1|f(\alp)|^8\d\alp \Bigr) 
\Bigl( \int_{\grM\setminus \grN}|f(\alp)|^5\d\alp \Bigr)^{r-1}\\
&\ll P^{s-5r-3}(P^5)(P^2L^{\eps-1/3})^{r-1}\ll P^{s-3r}L^{-1/4},
\end{align*}
and the proof of the lemma is complete.
\end{proof}

\begin{lemma}\label{lemma4.4} There is a positive number $\del$ such that $N(P;\grv)\ll P^{s-3r-\del}$.
\end{lemma}

\begin{proof} If $\bfalp\in \grv$, then for some index $j$ with $7\le j\le s$, one has
 $\gam_j(\bfalp)\not\in \grM+\dbZ$, and so $\bfalp\in \grm_j$. Thus, combining (\ref{2.3}) with a trivial
 estimate for $|f(\alp)|$, we find that for some suffix $j\in [7,s]$, and some tuple $(j_1,\ldots,j_{3r})$ with
$$1\le j_1<j_2<j_3\le 6<j_4<\ldots <j_{3r}\le s,$$
one has
\begin{equation}\label{4.4}
N(P;\grv)\ll P^{s-6r-1}\sup_{\bfalp\in \grm_j}|f(\gam_j(\bfalp))|\oint 
|g_{j_1}g_{j_2}g_{j_3}f_{j_4}\ldots f_{j_{3r}}|^2\d\bfalp .
\end{equation}
The matrix underlying the linear forms $\gam_{j_1},\ldots ,\gam_{j_{3r}}$ is highly non-singular, and so
 we may apply Theorem \ref{theorem3.3} to estimate the integral on the right hand side of (\ref{4.4}).
 Moreover, by Weyl's inequality (see \cite[Lemma 2.4]{Vau1997}), one has
$$\sup_{\bfalp \in \grm_j}|f(\gam_j(\bfalp))|\le \sup_{\bet \in \grm}|f(\bet)|\ll P^{3/4+\eps}.$$
We therefore conclude that for some positive number $\del$, one has
$$N(P;\grv)\ll P^{s-6r-1}(P^{3/4+\eps})(P^{3r+\xi+\eps})\ll P^{s-3r-\del}.$$
This completes the proof of the lemma.\end{proof}

By combining Lemmata \ref{lemma4.2}, \ref{lemma4.3} and \ref{lemma4.4}, we infer that whenever
 the system (\ref{1.1}) possesses a non-zero $p$-adic solution, one has
\begin{align*}
N(P)&=N(P;\grP)+N(P;\grV\setminus \grP)+N(P;\grv)\\
&\gg P^{s-3r}+O(P^{s-3r}L^{-1/4}+P^{s-3r-\del})\gg P^{s-3r}.
\end{align*}
This completes our proof of Theorem \ref{theorem1.1}.

We remark that the condition in Theorem \ref{theorem1.1} that $(c_{ij})$ be highly 
non-singular can certainly be relaxed. Let us refer to the number of columns lying in a given 
one dimensional subspace of the column space of $(c_{ij})$ as the {\it multiplicity} of that 
subspace. The discussion of \S\S2 and 3 would suffer no ill consequences were $(c_{ij})$ to 
satisfy the condition that the maximum multiplicity be $2$. In order to see this, one has 
simply to note that in such circumstances, the mean value estimates relevant to the application 
of the Hardy-Littlewood method can be related, via H\"older's inequality, to mean values of the 
shape (\ref{3.5}). We note in this context that the matrix $(c_{ij})$ occuring in Theorem 
\ref{theorem1.1} is of course different from that occuring in Theorem \ref{theorem3.3}. With 
rather greater effort in a more cumbersome argument, this maximum multiplicity $2$ could 
be increased to $3$, and even several multiplicities of $4$ can be tolerated. This and further 
refinements are topics that we intend to pursue on a future occasion.

\noindent {\bf Acknowledgement.} The authors are grateful to the referees for the 
extreme care taken in reviewing this paper, and in particular for numerous suggestions 
which have clarified our exposition and prompted significant corrections.

\bibliographystyle{amsbracket}

\begin{thebibliography}{18}

\bibitem{ABC1992}
O. D. Atkinson, J. Br\"udern and R. J. Cook, \emph{Simultaneous additive congruences to a large prime modulus}, Mathematika \textbf{39} (1992), 1--9.

\bibitem{Bak1989}
R. C. Baker, \emph{Diagonal cubic equations III}, Proc. London Math. Soc. (3) \textbf{58} (1989), 495--518. 

\bibitem{BC1992}
J. Br\"udern and R. J. Cook, \emph{On simultaneous diagonal equations and inequalities}, Acta Arith. \textbf{62} (1992), 125--149.

\bibitem{BW2002}
J. Br\"udern and T. D. Wooley, \emph{Hua's Lemma and simultaneous diagonal equations}, Bull. London Math. Soc. \textbf{34} (2002), 279--283.

\bibitem{BW2007}
J. Br\"udern and T. D. Wooley, \emph{The Hasse principle for pairs of diagonal cubic forms}, Ann. of Math. (2) \textbf{166} (2007), 865--895.

\bibitem{Dav2005}
H. Davenport, \emph{Analytic methods for Diophantine equations and Diophantine inequalities}, 2nd edition, Cambridge University Press, Cambridge, 2005.

\bibitem{DL1966}
H. Davenport and D. J. Lewis, \emph{Cubic equations of additive type}, Philos. Trans. Roy. Soc. London Ser. A \textbf{261} (1966), 97--136.

\bibitem{DL1969}
H. Davenport and D. J. Lewis, \emph{Simultaneous equations of additive type}, Philos. Trans. Roy. Soc. London Ser. A \textbf{264} (1969), 557--595.

\bibitem{Gow2001}
W. T. Gowers, \emph{A new proof of Szemer\'edi's theorem}, Geom. Funct. Anal. \textbf{11} (2001), 
no. 3, 465--588.

\bibitem{HB1998}
D. R. Heath-Brown, \emph{The circle method and diagonal cubic forms}, Phil. Trans. Roy. Soc. London Ser. A \textbf{356} (1998), 673--699.

\bibitem{Hoo1986}
C. Hooley, \emph{On Waring's problem}, Acta Math. \textbf{157} (1986), 49--97.

\bibitem{Hoo1996}
C. Hooley, \emph{On Hypothesis $K^*$ in Waring's problem}, Sieve methods, exponential sums, and their applications in number theory (Cardiff, 1995), pp. 175--185, Cambridge University Press, 1997.

\bibitem{Vau1986}
R. C. Vaughan, \emph{On Waring's problem for cubes}, J. Reine Angew. Math. \textbf{365} (1986), 122--170.

\bibitem{Vau1989}
R. C. Vaughan, \emph{A new iterative method in Waring's problem}, Acta Math. \textbf{162} (1989), 1--71.

\bibitem{Vau1997}
R. C. Vaughan, \emph{The Hardy-Littlewood method}, 2nd edition, Cambridge University Press, Cambridge, 1997.

\bibitem{Woo1991}
T. D. Wooley, \emph{On simultaneous additive equations, II}, J. Reine Angew. Math. \textbf{419} (1991), 141--198.

\bibitem{Woo2000}
T. D. Wooley, \emph{Sums of three cubes}, Mathematika \textbf{47} (2000), 53--61.

\end{thebibliography}
\providecommand{\bysame}{\leavevmode\hbox to3em{\hrulefill}\thinspace}

\end{document}